\documentclass{amsart}
\usepackage{amsfonts}
\usepackage{amsmath,amscd}
\usepackage{amssymb}
\usepackage{amsthm}
\usepackage{newlfont}

 \newtheorem{thm}{Theorem}[section]
 \newtheorem{cor}[thm]{Corollary}
 \newtheorem{lem}[thm]{Lemma}
 \newtheorem{prop}[thm]{Proposition}
 \theoremstyle{definition}
 
 \theoremstyle{remark}

 \numberwithin{equation}{section}
\begin{document}

\title[Non abelian tensor square of non abelian prime power groups]
 {Non abelian tensor square of non abelian prime power groups}

\author[P. Niroomand]{Peyman Niroomand}
\address{School of Mathematics and Computer Science\\
Damghan University, Damghan, Iran}
\email{niroomand@du.ac.ir, p\underline~ niroomand@yahoo.com}


%\subjclass{}


%\date{\today}

%\dedicatory{}

\begin{abstract}
%% Text of abstract
 For every
$p$-group of order $p^n$ with the derived subgroup of order $p^m$, Rocco in \cite{roc} has shown that the
order of tensor square of $G$ is at most $p^{n(n-m)}$.
In the present paper not only we improve his bound for non-abelian $p$-groups but also we describe the
structure of all non-abelian $p$-groups when the bound is
attained for a special case. Moreover, our results give as well an 
 upper bound for the order of  $\pi_3(SK(G, 1))$.

\end{abstract}

%%% ----------------------------------------------------------------------
\maketitle
%%% ----------------------------------------------------------------------
\section{Introduction and Preliminaries}

The tensor square $G\otimes G$ of a group $G$ is a group generated
by the symbols $g\otimes h$ subject to the relations
\[gg^{'}\otimes h=(^gg^{'}\otimes ^gh)(g\otimes h) ~\text{and}~g\otimes hh^{'}=(g\otimes h)\ (^hg\otimes h^{'}) \]
for all $g,g',h,h'\in G,$ where $^gg'=gg'g^{-1}$. The non abelian
tensor square is a special case of non abelian tensor product, which
was introduced by R. Brown and J.-L. Loday in \cite{br2}.

There  exists a homomorphism of groups $\kappa :G\otimes
G\rightarrow G^{'}$  sending $g\otimes h$ to $[g,h]=ghg^{-1}h^{-1}$.
The kernel of $\kappa$ is denoted by $J_2(G)$; its topological
interest is in the formula $\pi_3SK(G, 1) = J_2(G)$ (see
\cite{br2}).

According to the formula $\pi_3SK(G, 1) = J_2(G)$  computing the
order of $G\otimes G$ has interests in topology in addition to its
interpretation as a problem in the group theory.

Rocco in \cite{roc} and later Ellis in \cite{el} have shown that the
order of tensor square of $G$ is at most $p^{n(n-m)}$ for every
$p$-group of order $p^n$ with the derived subgroup of order $p^m$.

The purpose of this paper is a further investigation on the order of
tensor square of non abelian $p$-groups. We focus on non abelian
$p$-groups because in abelian case the non abelian tensor square
coincides with the usual abelian tensor square of abelian groups.
To be precise, for a non abelian $p$-group of order $p^n$ and  the derived subgroup of order $p^m$, 
we prove that $|G\otimes G|\leq p^{(n-1)(n-m)+2}$ and also we obtain the explicit structure of $G$  when $|G\otimes G|=p^{(n-1)^2+2}$.  It easily seen that the bound is less than of Rocco's bound, unless that
 $G\cong Q_8$ or $G\cong E_1$, which causes two
bounds to be equal. As a corollary by  using the fact $\pi_3SK(G, 1)\cong \mathrm{Ker}(G\otimes G\stackrel{\kappa}\rightarrow G^{'})$, we can see that $ |\pi_3SK(G, 1)| =| J_2(G)|\leq p^{n(n-m-1)+2}$.

Thorough the paper, $D_8$, $Q_8$ denote the dihedral and quaternion
group of order 8, $E_1$ and $E_2$ denote the extra-special p-groups
of order $p^3$ of exponent $p$ and $p^2$, respectively. Also
$C_{p^t}^{(k)}$ and $\nabla(G)$ denote the direct product of $k$
copies of the cyclic group of order $p^t$ and the subgroup generated
by $g\otimes g $ for all $g$ in $G$, respectively.

\section{Main Results}
The aim of this section is finding an upper bound for the order
tensor square of non abelian $p$-groups of order $p^n$ in terms of
the order of $G^{'}$. Also in the case for which $|G^{'}|=p$, the
structure of groups is obtained when $|G\otimes G|$ reaches the
upper bound. 

\begin{prop}\label{es}\cite[Proposition 9]{br}.
Given a central extension
\[1\rightarrow Z\longrightarrow H\longrightarrow G\longrightarrow1\]
there is an exact sequence
\[(Z\otimes H)\times(H\otimes Z)\stackrel{l}\longrightarrow H\otimes H\longrightarrow G\otimes G\longrightarrow 1\]
in which $\mathrm {Im}~l$ is central.
\end{prop}

\begin{prop}\cite[Proposition 13, 14]{br}\label{3}
The tensor square of $D_8$ and  $Q_8$ is isomorphic to
\[C_2^{(3)}\times C_4~\text{and}~C_2^{(2)}\times C_4^{(2)},\] respectively.

\end{prop}
Recall that \cite{br,br2} the order of tensor square of G is equal
to $|\nabla(G)||\mathcal{M}(G)||G^{'}|$, where $\mathcal{M}(G)$ is
the Schur multiplier of $G$.

Put $G^{ab}=G/G^{'}$. In analogy with the above proposition the
following lemma is characterized the tensor square of extra-special
$p$-groups of order $p^3 (p\neq2).$
\begin{lem}\label{2}
The tensor square of $E_1$ and $E_2$ are isomorphic to $C_p^{(6)}$
and $ C_p^{(4)}$, respectively.

\end{lem}
\begin{proof} It can be proved from \cite[Theorem 2]{el} that  $E_1\otimes E_1$ is elementary
abelian. Now, by invoking \cite[Proposition 2.2 (iii)]{ru},
$\nabla(E_1)\cong \nabla(E_1^{ab})$ and hence $|\nabla(E_1)|=p^3$.
On the other hand, \cite[Theorem 3.3.6]{ka} implies that the Schur
multiplier of $E_1$ is of order $p^2$, and so  $|E_1\otimes
E_1|=p^6$.

In the case $G=E_2$ in a similar fashion, we can prove that
$E_2\otimes E_2\cong E_2^{ab}\otimes E_2^{ab}$.
\end{proof}
\begin{lem}\cite[Corollary 2.3]{ni}\label{de} The tensor square of an
extra-special $p$-group $H$ of order $p^{2m+1}$ is elementary
abelian of order $p^{4m^2}$, for $m\geq 2$.
\end{lem}

 \begin{prop}\label{pro} Let $G$ be a $p$-group of order $p^n$ and $|G^{'}|=p$. If one of the following conditions holds, then the order of tensor square is less than  $p^{(n-1)^2+2}$.
\begin{itemize}
\item[(i)] $G^{ab}$ is not elementary abelian;
\item[(ii)] $G^{ab}$ is elementary abelian and $Z(G)$ is not elementary abelian.
\end{itemize}
\end{prop}
\begin{proof}[Proof (i)]The proof is an upstanding result of Proposition \ref{es} while $Z=G^{'}$.
 Let $G^{ab}=C_{p^{m_1}}\times C_{p^{m_2}}\times\ldots\times
C_{p^{m_k}}$ where $\sum_{i=1}^{k}m_i=n-1$ and $m_i \leq m_{i+1}$
for all $i$  $1\leq i\leq k-1$. Then
\[\begin{array}{lcl}|G\otimes G|&\leq&|G^{'}\otimes G^{ab}||G^{ab}\otimes G^{ab}|\vspace{.3cm}\\&=&|C_p\otimes
C_{p^{m_1}}\times C_{p^{m_2}}\times\ldots\times
C_{p^{m_k}}|\\&&|C_{p^{m_1}}\times C_{p^{m_2}}\times\ldots\times
C_{p^{m_k}}\otimes C_{p^{m_1}}\times C_{p^{m_2}}\times\ldots\times
C_{p^{m_k}}|\vspace{.3cm}\\&=& p^{m_k+\ldots +m_1+2(m_{k-1}+\ldots
+m_1+m_{k-2}+\ldots +m_1+\ldots +m_1)+k}\vspace{.3cm}\\&\leq&
p^{n-1+2(n-3+n-4+\ldots +n-2k+3)+k}\vspace{.3cm}\\&<&
p^{(n-1)^2+2},\end{array}\] as required.

\proof[(ii)] Since $G^{ab}$ is a vector space on $C_p$, let
$H/G^{'}$ be the complement of $Z(G)/G^{'}$ in $G^{ab}$. Moreover
$H$ is extra-special and $G=HZ(G)$. There is an epimorphism $H\times
Z(G)\otimes H\times Z(G)\longrightarrow G\otimes G$, so  \[|G\otimes
G|\leq|H\times Z(G)\otimes H\times Z(G)|.\] Let $|Z(G)|=p^k$ and
$|H|=p^{2m+1}$, we can suppose that $k\geq 2$ by using Proposition
\ref{3}. Now the following two cases can be considered.

Case $(i).$ First suppose that  $m\geq 2$.

Let $Z(G)\cong C_{p^{k_1}}\times\ldots \times C_{p^{k_t}}$ and
$\sum_{i=1}^{t}k_i=n-2m.$ Applying Lemma \ref{de} and
\cite[Proposition 11]{br}, we have
\[\begin{array}{lcl}|G\otimes G|&\leq&|H\otimes H|{|H\otimes Z(G)|}^2|Z(G)\otimes Z(G)|\vspace{.3cm}\\&=&
p^{4m^2}|C_p^{(m)}\otimes C_{p^{k_1}}\times\ldots \times
C_{p^{k_t}}|^{2}|C_{p^{k_1}}\times\ldots \times C_{p^{k_t}}\otimes
C_{p^{k_1}}\times\ldots \times
C_{p^{k_t}}|\vspace{.3cm}\\&=&p^{4m^2}p^{2mt}p^{(2t-1)k_1+(2t-3)k_2+\ldots
+k_t}\vspace{.3cm}\\&\leq&
p^{4m^2}p^{2mt}p^{n-2m+2(n-2m-2+\ldots+n-2m-t)}\vspace{.3cm}\\&<&
p^{(n-1)^2+2},
\end{array}\]as required.

Case $(ii).$ Without loss of generality,
 we can suppose that $Z(G)\cong C_{p^2}$. Now the result is obtained by using Proposition \ref{es} and the fact that $|Im l|\geq p$.
 \end{proof}

\begin{thm}\label{mt}Let $G$ be a non abelian $p$-group of order $p^n$. If $|G^{'}|=p$, then
\[|G\otimes G|\leq p^{(n-1)^2+2},\] and the equality holds if and only if $G$ is isomorphic to $H\times E$,
where  $H\cong E_1$ or $H\cong Q_8$ and  $E$ is an elementary
abelian $p$-group.

\end{thm}
\begin{proof}
  One can assume that $G^{ab}$ and $Z(G)$ are elementary abelian and $|Z(G)|\geq p^2$ by Proposition \ref{2}.
 Let $E$ be the complement of $G^{'}$ in $Z(G)$. Thus there exists an extra-special $p$-group $H$ of order $p^{2m+1}$ such that $G\cong H\times E$.

In the case $m\geq 2$, it is easily seen that $|G\otimes G|<
p^{(n-1)^2+2}$. For $m=1$,
\[|G\otimes G|=|H\otimes H||E\otimes E||E\otimes H|^2\]  where
 $|E\otimes E||E\otimes H|^2=p^{(n-1)(n-3)}$.

 Now Proposition \ref{3} and Lemma \ref{2} imply that
$|G\otimes G|=p^{(n-1)^2+2}$ when $H\cong Q_8$ or $H$ has exponent
$p$.
\end{proof}
\begin{thm} Let $G$ be a non abelian $p$-group of order $p^n$. If $|G^{'}|=p^m$, then
\[|G\otimes G|\leq p^{(n-1)(n-m)+2}.\]
\end{thm}
\begin{proof}
We prove theorem by induction  on $m$. For $m=1$ the result is
obtained by Theorem \ref{mt}.

Let $m\geq 2$ and $K$ be a central subgroup of order $p$ contained
in $G^{'}$. Induction hypothesis and Proposition \ref{3} yield
\[\begin{array}{lcl}|G\otimes G|&\leq&|K\otimes G^{ab}||G/K\otimes
G/K|\vspace{.3cm}\\&\leq&
p^{n-m}p^{(n-m)(n-2)+2}=p^{(n-1)(n-m)+2}.\end{array}\]
 \end{proof}
 \begin{cor}
 Let $G$ be a non abelian $p$-group of order $p^n$. If $|G^{'}|=p^m$, then
\[|\pi_3SK(G, 1)| \leq p^{n(n-m-1)+2}.\]
In particular when $m=1$, then
\[|\pi_3SK(G, 1)|\leq p^{n(n-2)+2},\] and the equality holds if and only if $G$ is isomorphic to $H\times E$,
in which  $H$ is extra-special of order $p^3$ of exponent $p$  or
$H\cong Q_8$ and  $E$ is an elementary abelian $p$-group.
 \end{cor}
\begin{cor} If the order of tensor square of $G$ is equal to $p^{(n-1)^2+2}$, then
\[G\otimes G\cong {C_p}^{((n-1)^2+2)}~(p\neq 2)~\text{or}~ G\otimes G\cong C_4^{(2)}\times C_2^{((n-1)^2-2)}.\]
\end{cor}


\begin{thebibliography}{20}
\bibitem{ru} R.D. Blyth, F. Fumagalli andM. Morigi. Some structural results on the non-abelian tensor square of
groups. \textit{J. Group Theory.} \textbf{13}(1) (2010), 83--94.
\bibitem{br} R. Brown, D.L. Johnson and E.F. Robertson. Some Computations of non Abelian tensor products of
groups. \textit{J. Algebra.}  \textbf{111} (1987), 177-202.
\bibitem{br2}  R. Brown and J.-L Loday. Van Kampen theorems for diagrams of spaces. With an appendix by M. Zisman, \textit{Topology}
 \textbf{26} (1987), 311-335.
\bibitem{el} G. Ellis. On the tensor square of a prime power group. \textit{Arch. Math.} \textbf{66} (1996), 467-469.
\bibitem{ka} G. Karpilovsky. \textit{The Schur multiplier} $($London Math. Soc. Monogr. $($N.S.$)$ 2 1987$)$.
\bibitem{ni} P. Niroomand, M. R. R. Moghaddam.  Some properties on the specific subgroup of tensor square, to appear in \textit{Comm. Algebra.}
\bibitem{roc} N. R. Rocco. On a construction related to the nonabelian tensor square of a group. \textit{Bol. Soc.
Brasil. Mat.} \textbf{22}(1) (1991), 63-79.
\bibitem{ro}  N. R. Rocco. A presentation for crossed embeding of finite solvable
groups. \textit{Comm. Algebra.}  \textbf{22}(6) (1994), 1975-1998.

\end{thebibliography}
\end{document}